\documentclass[12pt,leqno]{amsart}

\usepackage{amsmath,amsthm,amsfonts,latexsym,amssymb}

\newtheorem{thm}{Theorem}

\newtheorem{lem}[thm]{Lemma}

\newtheorem{cor}[thm]{Corollary}
\numberwithin{thm}{section}
\numberwithin{equation}{section}

\theoremstyle{definition}

\newcommand{\rat}{\mathbb Q}

\newcommand{\com}{\mathbb C}
\newcommand{\alg}{\overline\rat}
\newcommand{\proj}{\mathbb P}
\newcommand{\intg}{\mathbb Z}

\newcommand{\x}{{\bf x}}
\newcommand{\m}{{\bf m}}
\newcommand{\xij}{x_{i_0j_0}}
\newcommand{\dij}{d_{i_0j_0}}

\newcommand{\mij}{m_{i_0j_0}}

\title{Lower bounds on the projective heights of algebraic points}
\author{Charles L. Samuels}
\address{Department of Mathematics, University of Texas at Austin, 1 University Station C1200
  Austin, TX 78712}
\email{csamuels@math.utexas.edu}

\begin{document}

\begin{abstract}
  If $\alpha_1,\ldots,\alpha_r$ are algebraic numbers such that $$N=\sum_{i=1}^r\alpha_i
  \ne \sum_{i=1}^r\alpha_i^{-1}$$ for some integer $N$, then a theorem of Beukers and Zagier 
  \cite{BeukersZagier}
  gives the best possible lower bound on $$\sum_{i=1}^r\log h(\alpha_i)$$
  where $h$ denotes the Weil Height.  We will extend
  this result to allow $N$ to be any totally real algebraic number.  Our generalization
  includes a consequence of a theorem of Schinzel \cite{Schinzel} which bounds the height of
  a totally real algebraic integer.
\end{abstract}

\baselineskip=7mm
\maketitle
\section{Introduction}

Let $K$ be any number field and $v$ a place of $K$ extending the place $p$ of $\rat$.
Let $K_v$ denote the completion of $K$ at $v$ and $\rat_p$ the completion of $\rat$ at $p$.
Write $D_v=[K_v:\rat_p]$ and $D=[K:\rat]$ for the local and global degrees.
Let $\|\cdot\|_v$ be the unique absolute value on $K_v$ that extends the $p$-adic absolute
value on $\rat_p$ and define $|\cdot|_v=\|\cdot\|_v^{D_v/D}$.  
We note that $|\cdot|_v$ satsifies the product formula
\begin{equation*}
  \prod_v|\alpha|_v=1
\end{equation*}
for all $\alpha\in K^{\times}$ where the sum is taken over all places $v$ of $K$.  

Define the {\it projective height} of a point $\x=(x_0,\ldots,x_n)\in\proj^n(K)$ by
\begin{equation*}
  \log H(\x) = \sum_v\log\max_i|x_i|_v
\end{equation*}
and note that by the product formula $H$ is well defined on $\proj^n(K)$.
By our choice of absolute values, the definition of $H$ does
not depend on $K$ and therefore defines
a function on $\proj^n(\alg)$.  We define the {\it Weil Height} $h(\alpha)$ of a point 
$\alpha\in\alg$ by $h(\alpha) = H((1,\alpha))$.

By Kronecker's Theorem, $\log h(\alpha)\geq 0$ with equality if and only if $\alpha$ 
is $0$ or a root of unity.  
In 1933, Lehmer \cite{Lehmer} asked whether there exists a constant $\rho>1$ such that 
\begin{equation} \label{LehmersConjecture}
  \deg(\alpha)\log h(\alpha)\geq \log\rho
\end{equation}
in all other cases.  In particular, he asked whether we may take $\rho$ to be
the larger real root of $x^{10}+x^9-x^7-x^6-x^5-x^4-x^3+x+1$. 

Lehmer's problem is still open today though an affirmative answer has been given for certain classes
of algebraic numbers.  Smyth \cite{Smyth} proved that if $\alpha\ne 0$ and the minimal polynomial of $\alpha$
is not reciprocal then \eqref{LehmersConjecture} holds with $\rho$ the real root of $x^3-x-1$.  
It is a consequence of a theorem of Schinzel \cite{Schinzel} that if $\alpha$ is totally real we may take 
$\log\rho=\frac{1}{2}\log\frac{1+\sqrt 5}{2}$.  If we further assume that $\alpha$ is an algebraic
integer then the same bound holds without $\deg(\alpha)$ appearing on the left 
hand side of \eqref{LehmersConjecture}.  In this case, Schinzel's lower bound is best
possible by taking $\alpha=\frac{1+\sqrt 5}{2}$.  The best unconditional result
toward answering Lehmer's problem is a theorem of Dobrowolski \cite{Dobrowolski} which gives
a lower bound on $\deg(\alpha)\log h(\alpha)$ which tends to $0$ slowly as $\deg(\alpha)\to\infty$.

In a slightly different direction, Zhang \cite{Zhang} showed that there exists $\rho> 1$
such that
\begin{equation} \label{ZhangsTheorem}
  \log h(\alpha)+\log h(1-\alpha)\geq\log\rho
\end{equation}
whenever $\alpha$ is not $0$, $1$ or a primitive $6$th root of unity.
Zagier \cite{Zagier} used elementary methods to show that \eqref{ZhangsTheorem} holds with
$\log\rho=\frac{1}{2}\log\frac{1+\sqrt 5}{2}$ with cases of equality identified.  As Zagier notes,
it is interesting that this is the same lower bound that appears in Schinzel's bound \cite{Schinzel}
on the height of a totally real algebraic integer.
Our goal is to show that the results of Schinzel and Zagier are in fact 
consequences of a more general theorem.

Our proof will apply the methods of Beukers and Zagier \cite{BeukersZagier} who 
generalized the results of \cite{Zagier} in
the following way.  Let $\alpha_1,\ldots,\alpha_r$ be non-zero algebraic numbers such that
$\alpha_1+\cdots +\alpha_r=N$ and $\alpha_1^{-1}+\cdots +\alpha_r^{-1}\ne N$ for some 
integer $N$.  Then
\begin{equation} \label{BZBound}
  \sum_{i=1}^r\log h(\alpha_i) \geq \frac{1}{2}\log\frac{1+\sqrt 5}{2}
\end{equation}
with cases of equality.   We will further generalize this theorem
so that $N$ may be any totally real algebraic integer.  Then by taking $r=1$ we are
able to recover Schinzel's result.

\section{Main Results}
Suppose that $r, n_1, \ldots, n_r$ are positive integers and $K$ is a field.  Then we write
$\mathcal P(K)=\proj^{n_1}(K)\times\cdots\times\proj^{n_r}(K)$ and denote the coordinates
by $\x=(\x_0,\ldots,\x_r)$ with $\x_i=(x_{i0},\ldots,x_{in_i})$.  If $\x$ has
$x_{ij}\ne 0$ for all $i,j$ let $\x^{-1}$ be the point
obtained by replacing each coordinate $x_{ij}$ of $\x$ with $x_{ij}^{-1}$.
Following \cite{BeukersZagier}, choose any subset $I$ of $\{i|n_i=1\}$ and let $E=\{(i,0)|i\in I\}$.  
We refer to $E$ as the set of {\it exceptional index pairs}.  Index pairs not in $E$ are called
{\it regular index pairs}.
If a regular index pair appears in a monomial of a polynomial $Q(\x)$, then we say the monomial
is a {\it regular monomial} of $Q$.  Otherwise,  the monomial is called an {\it exceptional monomial}.
Also write $\|Q\|_v$ to denote the sum of the $v$-adic absolute values 
(using $\|\cdot\|_v$) of the coefficients of $Q$.

Let $F$ be a multihomogeneous polynomial over $\alg$ of multidegrees $d_1,\ldots,d_r$
so that $F$ defines a zero set in $\mathcal P(\alg)$.  The degree of $F$ in the variable $x_{ij}$ is
denoted $d_{ij}$ and define $\tilde d_i=-d_i+\sum_jd_{ij}$.  Then we set
\begin{equation*}
  \delta=\max\left\{\max_{i\in I}\left\{\frac{\tilde d_i+d_{i1}}{n_i+1}\right\}, 
  \max_{i\not\in I}\left\{\frac{\tilde d_i}{n_i+1}\right\}\right\}
\end{equation*}
and assume that $F$ has the following properties:
\begin{itemize}
  \item[{\it (i)}] the coefficients of $F$ are totally real algebraic integers
  \item[{\it (ii)}] the coefficients of regular monomials of $F$ are integers.
\end{itemize}
Then for $v$ Archimedean define
\begin{equation*}
  c(F,v,i,j)=\left\|\frac{\partial F}{\partial x_{ij}}\right\|_v.
\end{equation*}

In \cite{BeukersZagier}, Beukers and Zagier consider only polynomials $F$ having integer coefficients,
so clearly $c(F,v,i,j)$ does not depend on the place $v$.  In fact, $c(F,v,i,j)$ is defined in
\cite{BeukersZagier} using the usual absolute value on the complex numbers rather than $\|\cdot\|_v$.
Since we assume only the weaker conditions
{\it (i)} and {\it (ii)}, $c(F,v,i,j)$ may indeed depend on $v$ as the notation suggests.
Therefore, we require the absolute value $\|\cdot\|_v$ in this definition.

However, in the special case that $(i,j)\not\in E$, $c(F,v,i,j)$ depends only on the 
regular monomials of $F$.  So by property {\it (ii)}, $c(F,v,i,j)$ depends only on the monomials 
of $F$ having integer coefficients, and therefore, does not depend on $v$.  Then we may define
\begin{equation*}
  C_F = C_F(E)= \max_{(i,j)\not\in E}c(F,v,i,j)
\end{equation*}
and by our remarks above, $C_F$ does not depend on $v$.  
We now state our main theorem which is a direct generalization of the main theorem in \cite{BeukersZagier}.

\begin{thm} \label{main}
  Let $F$ be a multihomogeneous polynomial satisfying properties {\it (i)} and {\it (ii)}
  above for some exceptional set $E$.
  If $\x\in \mathcal P(\alg)$ is such that $F(\x)=0$, $\prod_{i,j}x_{ij}\ne 0$ and 
  $F(\x^{-1})\ne 0$ then
  \begin{equation*}
    \sum_{i=1}^{r}(n_i+1)\log H(\x_i) \geq \log\rho
  \end{equation*}
  where $\rho$ is the unique real root larger than $1$ of 
  $x^{-2}+C_F^{-1}x^{-\delta} =1$.
\end{thm}

Once again, we note that our theorem generalizes \cite{BeukersZagier} in that we allow the
coefficients of $F$ to come from a potentially larger set.  While the main theorem
in \cite{BeukersZagier} requires these coefficients to be integers, we allow some
of them to be any totally real algebraic integers.

Before we prove Theorem \ref{main} we demonstrate its relationship to our problem.  
Consider $r$ non-zero
algebraic numbers $\alpha_1,\ldots,\alpha_r$ such that $\alpha_1+\cdots +\alpha_r=N$
and $\alpha_1^{-1}+\cdots +\alpha_r^{-1}\ne N$.  
Corollary 2.1 of \cite{BeukersZagier} gives a lower bound on $\sum_{i=1}^r\log h(\alpha_i)$
when $N$ is an integer.  We apply Theorem \ref{main} to prove a direct generalization of
this result.

\begin{cor} \label{SumIsN}
  Suppose $\alpha_1,\ldots ,\alpha_r$ are non-zero algebraic numbers and $N$ is a
  totally real algebraic integer.  If $\alpha_1+\cdots +\alpha_r=N$ and
  $\alpha_1^{-1}+\cdots +\alpha_r^{-1}\ne N$ then
  \begin{equation*}
    \sum_{i=1}^r\log h(\alpha_i) \geq \frac{1}{2}\log\frac{1+\sqrt 5}{2}
  \end{equation*}
  with equality when $r=1$ and $\alpha_1=\frac{1+\sqrt 5}{2}$.
\end{cor}
\begin{proof}
  Write $\alpha_i=\alpha_{i1}$ for all $i$ and suppose that the $\alpha_{i0}$ are algebraic numbers.
  We consider the point 
  \begin{equation*}
    {\bf \alpha} = (\alpha_{10},\alpha_{11})\times\cdots\times(\alpha_{r0},\alpha_{r1})
    \in(\mathbb P^1(\alg))^r.
  \end{equation*}
  We will apply Theorem \ref{main} to this point with $I=\{1,\ldots,r\}$ so we
  have $E=\{(1,0),\ldots,(r,0)\}$.  
  Let $F$ be the homogeneous version of $x_{10}+\cdots+x_{r0}-N$.  That is,
  \begin{equation*}
    F(\x)=\sum_{i=1}^r x_{i1}\prod_{j\ne i}x_{j0} -N \prod_{j}x_{j0}
  \end{equation*}
  and note that $F$ satisfies properties {\it (i)} and {\it (ii)}.
  It is clear that $c(F,v,i,j)=1$ for all $(i,j)\not\in E$
  so that $C_F=1$.  We also have $n_i=1$, $d_i=1$ and $d_{i1}=1$ so that
  $\delta =1$.  Then by Theorem \ref{main}
  \begin{equation*}
    \sum_{i=1}^r 2\log H(\alpha_{i0},\alpha_{i1}) \geq \log\rho
  \end{equation*}
  where $\rho$ is the real root larger than $1$ of $x^{-2}+x^{-1}=1$.  Setting
  $\alpha_{i0} =1$ for all $i$ the result follows and the case of equality is clear.
\end{proof}

Note that the case of equality in Corollary \ref{SumIsN} is not unique.  For example, we also
have equality when $r=2$, $\alpha_1=1$ and $\alpha_2=\frac{1+\sqrt 5}{2}-1$.  Several
other cases of equality are given in \cite{BeukersZagier} and \cite{Zagier} using
integer values for $N$.

In the special case that $r=1$ Corollary \ref{SumIsN} implies that
$\log h(\alpha)\geq\frac{1}{2}\log\frac{1+\sqrt 5}{2}$ for all totally real
algebraic integers $\alpha\not\in\{0,\pm 1\}$.  
Therefore, Schinzel's bound \cite{Schinzel} on the height of a totally real algebraic integer
is a corollary of our result.
\begin{cor}
  If $\alpha$ is a totally real algebraic integer with $\alpha\not\in\{\pm 1,0\}$, then
  \begin{equation*}
    \log h(\alpha)\geq\frac{1}{2}\log\frac{1+\sqrt 5}{2}.
  \end{equation*}
\end{cor}

\section{Proof of Theorem \ref{main}}
We begin with some additional notation.
Recall that for a point $\x\in\mathcal P(K)$ for some field $K$ we denote the coordinates
$\x=(\x_1,\ldots,\x_r)$ with $\x_i=(x_{i0},\ldots,x_{in_i})$.  Similarly, for
a point $\m\in\intg^{n_1+1}\times\cdots\times\intg^{n_r+1}$ we set
$\m=(\m_1,\ldots,\m_r)$ with $\m_i=(m_{i0},\ldots,m_{in_i})$.
Define the product $\x^\m= \prod_{i,j}x_{ij}^{m_{ij}}$ and the set
\begin{equation*}
  M=\left\{\m\in\intg^{n_1+1}\times\cdots\times\intg^{n_r+1}\left|
  m_{ij}\geq 0,\ \sum_jm_{ij}=d_i\ \forall i\right.\right\}
\end{equation*}
so that the polynomial $F$ may be written
\begin{equation*}
  F(\x)=\sum_{\m\in M}s_{\m}\x^\m
\end{equation*}
where the $s_\m$ are totally real algebraic integers.  Let $\{G_k(x)\}$ be a finite
collection of multihomogeneous polynomials over $K$ with algebraic integer
coefficients.  Assume that $G_k$ has multidegrees $d_{k1},\ldots,d_{kr}$.
As above, we define the sets
\begin{equation*}
  M_k=\left\{\m\in\intg^{n_1+1}\times\cdots\times\intg^{n_r+1}\left|
  m_{ij}\geq 0,\ \sum_jm_{ij}=d_{ki}\ \forall i\right.\right\}
\end{equation*}
and write 
\begin{equation*}
  G_k(\x)=\sum_{\m\in M_k}s_{k\m}\x^\m
\end{equation*}
where the $s_{k\m}$ are algebraic integers.

If $K$ is a number field containing the coefficients of the polynomials
$G_k$ and $v$ is a place of $K$
we write $X(K)$ to denote the zero set of $F$ in $\mathcal P(K)$ and
$X(K_v)$ for the zero set of $F$ in $\mathcal P(K_v)$.  Let 
\begin{equation*}
  \Delta_v(K)=\{\x\in \mathcal P(K)\mid \|x_{ij}\|_v\leq 1\ \forall i,j\}
\end{equation*} 
and
\begin{equation*}
  \Delta(K_v)=\{\x\in \mathcal P(K_v)\mid \|x_{ij}\|_v\leq 1\ \forall i,j\}.
\end{equation*}
Then define
$X_v(K)_1=X(K) \cap \Delta_v(K)$ and $X(K_v)_1=X(K_v)\cap \Delta(K_v)$ and observe
that $X_v(K)_1\subset X(K_v)_1$.  Our first Lemma is an analog of Lemma 3.1 of
\cite{BeukersZagier}.

\begin{lem} \label{InitBound}
Suppose that $K$ is any number field containing the coefficients of the polynomials $G_k$, that
$v$ indexes the places of $K$, and that $a_k\geq 0$ for all $k$.  Set
\begin{equation*}
  w_i=\sum_ka_kd_{ki},\hspace{6mm}
  \log\lambda_v=-\max_{\x\in X(K_v)_1}\left\{\sum_ka_k\log\left\|
      G_k(\x)\right\|_v\right\}.
\end{equation*}
If $\x\in X(K)$ with $\prod_kG_k(\x)\ne 0$ then
\begin{equation*}
  \sum_{i=1}^r w_i\log H(\x_i) \geq \sum_{v|\infty}\frac{D_v}{D}\log\lambda_v.
\end{equation*}
\end{lem}
\begin{proof}
  We will prove that the local inequality
  \begin{align} \label{LocalInequality}
    \sum_{i=1}^r  w_i &\log(\max_j\|x_{ij}\|_v) \nonumber \\ 
    & \geq \sum_k a_k\log\|G_k(\x)\|_v+\left\{
      \begin{array}{ll}
        \log \lambda_v & \mathrm{if}\ v\mid\infty \\
        0 & \mathrm{if}\ v\nmid\infty
      \end{array}\right.
  \end{align}
  holds for all places $v$ of $K$.  

  We first assume that $v\nmid\infty$. Since each coefficient $s_{k\m}$ of $G_k$
  is an algebraic integer, we have that $\|s_{k\m}\|_v\leq 1$ for all $k, \m$.
  By the strong triangle inequality, there exists $\m\in M$ such that
  \begin{align*}
    \sum_k a_k\log\|G_k(\x)\|_v & \leq \sum_ka_k\log\|\x^\m\|_v \\
    & = \sum_ka_k\sum_id_{ki}\log\max_j\|x_{ij}\|_v\\
    & = \sum_iw_i\log\max_j\|x_{ij}\|_v
  \end{align*}
  and we have established \eqref{LocalInequality} in the case that $v\nmid\infty$.

  Next we assume that $v|\infty$.  For each $i$, let $j_0=j_0(i)$ be
  such that $\max_j\|x_{ij}\|_v = \|x_{ij_0}\|_v$.  Let
  $\x'$ be the point obtained by replacing each coordinate of $\x$ with $x_{ij}/x_{ij_0}$. 
  We have that $\|x_{ij}/x_{ij_0}\|_v\leq 1$ for all $i, j$ so that
  \begin{equation*}
    \sum_{i=1}^r w_i\log\max_j\left\|\frac{x_{ij}}{x_{ij_0}}\right\|_v
    \geq \sum_ka_k\log\left\|G_k\left(\x'\right)\right\|_v + \log\lambda_v
  \end{equation*}
  By the homogeneity of the polynomials $G_k$ we find that
  \begin{align*}
    \sum_ka_k\log\left\|G_k\left(\x'\right)\right\|_v
    & = \sum_ka_k\log\left\|\prod_ix_{ij_0}^{-d_{ki}}G_k(\x)\right\|_v \\
    & = \sum_ka_k\log\|G_k(\x)\|_v-\sum_i w_i\log\|x_{ij_0}\|_v
  \end{align*}
  and conclude that
  \begin{equation*} 
    \sum_{i=1}^r w_i\log(\max_j\|x_{ij}\|_v) \geq \sum_k a_k\log\|G_k(\x)\|_v + \log\lambda_v
  \end{equation*}
  so we have established \eqref{LocalInequality}.  Now sum both sides of \eqref{LocalInequality}
  over all places $v$ of $K$ and apply the product formula.  The desired result follows.
\end{proof}
  
Note that in the version of Lemma \ref{InitBound} that appears in \cite{BeukersZagier}, 
the polynomials $G_k$ are assumed to have integer coefficients.
Therefore, each $\lambda_v$ is in fact independent of $v$.  In this simpler situation, 
Beukers and Zagier define $\lambda_v$ using the usual absolute value on $\com$ rather than
$\|\cdot\|_v$ on $K_v$.  

In our version of Lemma \ref{InitBound} we allow for the $G_k$ to 
have any algebraic integer coefficients, so we must define $\lambda_v$ using $\|\cdot\|_v$
on a number field containing the coefficients of the $G_k$.  It is certainly possible that
$\lambda_v$ does indeed depend on the place $v$.
However, with appropriate choices for $G_k$ and $a_k$, conditions ${\it (i)}$ and ${\it (ii)}$ 
are enough to produce a universal lower bound on $\lambda_v$ that does not depend on $v$.
In view of Lemma \ref{InitBound}, this lower bound gives a bound on $\sum_{i=1}^{r}(n_i+1)\log H(\x_i)$.

Before we make selections for the $G_k$ and the $a_k$, we state Lemmas 3.2 and 3.3 
of \cite{BeukersZagier} for use later.  Although the statement of Lemma 3.2 in \cite{BeukersZagier}
is for polynomials with integer coefficients, it is easily
verified that the lemma holds for polynomials with complex coefficients and we state
this generalization here.

\begin{lem} \label{UnitCircle}
  Suppose $v$ is an Archimedean place of a number field $K$ with $D_v=2$.
  If $Q_k(\x)$ are multihomogeneous polynomials with coefficients in $K_v$ then
  the function  $\sum_ka_k\log\|Q_k(\x)\|_v$ assumes a maximum in $X(K_v)_1$
  at a point $\x$.  Moreover, $\x$ has one coordinate pair $(i_0,j_0)$ such that
  $\|x_{ij}\|_v=1$ for all $(i,j)\ne(i_0,j_0)$.
\end{lem}

\begin{lem} \label{Calculus}
  let $\alpha,\beta,\gamma > 0$  Let $l$ be the unique minimum of the function
  \begin{equation*}
    u\log\frac{\gamma u}{u+v}+v\log\frac{v}{u+v}
  \end{equation*}
  under the constraints $u,v\geq 0$, $\alpha u+\beta v =1$.  Then $e^{-l}$ is the unique real
  root larger than $1$ of $\gamma^{-1}x^{-\alpha}+x^{-\beta}=1$.
\end{lem}

We now make our selections for $G_k$ and $a_k$ following \cite{BeukersZagier}.
For $G_k$ we choose the coordinates $x_{ij}$ and the polynomial
\begin{equation*}
  \tilde F(\x) = F(\x^{-1})\prod_{i,j}x_{ij}^{d_{ij}}.
\end{equation*}
Note that $\tilde F$ is multihomogeneous with multidegrees given by
$\tilde d_i=-d_i+\sum_jd_{ij}$.  Write $a_{ij}$ and $b$ for the
values of $a_k$ corresponding to $x_{ij}$ and $\tilde F$, respectively.  
In this situation we have
\begin{equation} \label{lambda}
  \log\lambda_v = -\max_{\x\in X(K_v)_1}\left\{
    \sum_{i,j}a_{ij}\log\|x_{ij}\|_v+b\log\|\tilde F(\x)\|_v\right\}.
\end{equation}
Finally, let $\rho$ be the real root greater than $1$ of $x^{-2}+C_F^{-1}x^{-\delta} =1$.

\begin{lem} \label{FindLambda}
  Suppose $K$ is a number field containing the coefficients of $F$ and $D_v=2$ for all
  Archimedean places $v$ of $K$.
  Then there exist $a_{ij},b\geq 0$ such that $n_i+1=\sum_ja_{ij}+b\tilde d_i$ for
  all $i$ and $\lambda_v\geq\rho$ for all $v|\infty$.
\end{lem}
\begin{proof}
  Following \cite{BeukersZagier}, we define each $a_{ij}$ in terms of $b$ by
  \begin{equation} \label{ANotInI}
    a_{ij} = 1-\frac{\tilde d_i}{n_i+1}b \hspace{3mm} \mathrm{if}\ i\not\in I
  \end{equation}
  and
  \begin{equation} \label{AInI}
    a_{ij}= 1-\frac{\tilde d_i+(-1)^jd_{i1}}{n_i+1}b \hspace{3mm} \mathrm{if}\ i\in I
  \end{equation}
  so we immdiately have $n_i+1=\sum_ja_{ij}+b\tilde d_i$.  Now we need only select $b$
  so that $\lambda_v\geq\rho$.
  
  We will show that under the assumptions \eqref{ANotInI} and \eqref{AInI} 
  \begin{equation} \label{FindLambdaBound1}
    -\log\lambda_v \leq  b\log\frac{2bC_F}{(1-\delta b)+2b}+\frac{1-\delta b}{2}\log
    \frac{1-\delta b}{(1-\delta b)+2b}
  \end{equation}
  holds for every Archimedean place $v$ of $K$.  Let
  \begin{equation*}
    \Phi(\x) = \sum_{i,j}a_{ij}\log\|x_{ij}\|_v+b\log\|\tilde F(\x)\|_v
  \end{equation*}
  so that we must give an upper bound on $-\log\lambda_v=\max_{\x\in X(K_v)_1}\Phi(\x)$.
  By Lemma \ref{UnitCircle} this maximum is attained at a point $\x\in X(K_v)_1$ where
  $\|x_{i_0j_0}\|_v\leq 1$ for some coordinate pair $(i_0,j_0)$ and $\|x_{ij}\|_v=1$
  for all $(i,j)\ne (i_0,j_0)$.  Hence, $\bar x_{ij}=x_{ij}^{-1}$ for all $(i,j)\ne (i_0,j_0)$.
  Moreover, $\Phi(\x)\to -\infty$ as $x_{ij}\to 0$ for any $i,j$.  Therefore, we must have
  $x_{i_0j_0}\ne 0$ so that the point $\x^{-1}$ is well defined.

  Suppose first that $(i_0,j_0)\not\in E$ and write $x=\xij$, $d=\dij$ and $m=\mij$
  for any $\m\in M$.  
  Let $\bar \x$ be the point obtained by replacing each coordinate of $\x$ with $\bar x_{ij}$.
  By property $(i)$, the coefficients of $F$ are in the fixed field of complex 
  conjugation in $K_v$.  Using $F(\x)=0$ we obtain
  \begin{align*}
    F(\x^{-1})
    & = F(\x^{-1}) - F(\bar \x) \\ 
    & = \sum_{\m\in M}s_\m \x^{-\m} - \sum_{\m\in M}s_\m \bar \x^{\m} \\
    & = \sum_{\m\in M}s_\m\left(\frac{\bar \x^\m}{\bar x^m}\right)(x^{-m}-\bar x^m)
  \end{align*}
  and note that $\|\bar \x^\m/\bar x^m\|_v=1$ for all $\m\in M$.  We now apply the triangle
  inequality to find
  \begin{align*}
    \|\tilde F(\x)\|_v
    & \leq \|x\|_v^d\sum_{\m\in M}\|s_\m(x^{-m}-\bar x^m)\|_v \\
    & \leq \|x\|_v^d\cdot\|x^{-1}-\bar x\|_v\sum_{\m\in M}m\|s_\m\|_v\cdot\|x^{-1}\|_v^{m-1} \\
    & \leq \|x\|_v^d\cdot\|x^{-1}-\bar x\|_v\cdot\|x\|_v^{1-d}\sum_{\m\in M}m\|s_\m\|_v \\
    & = c(F,v,i_0,j_0)(1-\| x\|_v^2),
  \end{align*}
  where the last equality follows since $c(F,v,i_0,j_0) = \sum_{\m\in M}m\|s_\m\|_v$.
  Now let $\xi=\| x\|_v^2$, $c=c(F,v,i_0,j_0)$ and $a=a_{i_0j_0}$.  We have
  \begin{equation*} \label{Diff}
    -\log\lambda_v \leq \max_{\xi\in[0,1]}\left(b\log(c(1-\xi))+\frac{a}{2}\log \xi\right).
  \end{equation*}
  Differentiating we find that this maximum is attained at $\xi_0 = a/(a+2b)$
  and its value is
  \begin{equation} \label{Bound1}
     b\log\frac{2bc}{a+2b}+\frac{a}{2}\log\frac{a}{a+2b}.
  \end{equation}
  By definition $a = 1-b\tilde d_i/(n_i+1)\geq1-\delta b$.
  Therefore \eqref{Bound1} is bounded above by
  \begin{equation*}
    b\log\frac{2bC_F}{(1-\delta b)+2b}+\frac{1-\delta b}{2}\log
    \frac{1-\delta b}{(1-\delta b)+2b}
  \end{equation*}
  and \eqref{FindLambdaBound1} follows.
  
  Next assume that $(i_0,j_0)\in E$ so that $j_0=0$.  We have that $\|x_{i_00}\| \leq 1$ and
  $\|x_{ij}\| = 1$ for all $(i,j)\ne(i_0,0)$.  We write $x=x_{i_00}$, $x'=x_{i_01}$,
  $d=d_{i_00}$, $d'=d_{i_01}$, $m=m_{i_00}$ and $m'=m_{i_01}$ for each $\m\in M$.
  Then we find
  \begin{align*}
    \|\tilde F(\x)\|_v
    & = \|x^dF(\x^{-1}) - \bar x^{-d}F(\bar\x)\|_v \\
    & = \left\|\sum_{\m\in M}s_\m\left(\frac{\bar\x^\m}{\bar x^m}\right)(x^{d-m}-\bar x^{m-d})\right\|_v \\
    & \leq \sum_{\m\in M}\|s_\m(x^{d-m}-\bar x^{m-d})\|_v \\
    & \leq \|x-\bar x^{-1}\|_v\cdot\sum_{\m\in M}(d-m)\cdot\|s_\m\|_v\cdot\|x^{-1}\|_v^{d-m-1}
  \end{align*}
  We know that $m+m'=d_{i_0}\geq d$ so $d-m\leq m'$.  Therefore, we obtain
  \begin{align*}
    \|\tilde F(\x)\|_v
    & \leq \|x-\bar x^{-1}\|_v\sum_{\m\in M}m'\|s_\m\|_v\cdot\|x\|_v^{1-m'} \\
    & \leq \|x-\bar x^{-1}\|_v\cdot\|x\|_v^{1-d'}\sum_{\m\in M}m'\|s_m\|_v \\
    & = (1-\|x\|_v^2)\cdot\|x\|_v^{-d'} c(F,v,i_0,1)
  \end{align*}
  Let $\xi=\|x\|_v^2$ and $c=c(F,v,i_0,1)$ so that
  \begin{equation} \label{SecondLambda}
    \log\lambda_v \leq \max_{\xi\in[0,1]}\left(b\log(c(1-\xi))-
      \frac{d_{{i_0}1}b}{2}\log\xi+\frac{a_{{i_0}0}}{2}\log \xi\right).
  \end{equation}
  With $a=a_{{i_0}0}-d'b$ we have that the right hand 
  side of \eqref{SecondLambda} equals
  \begin{equation*}
    b\log\frac{2bc}{a+2b}+\frac{a}{2}\log\frac{a}{a +2b}.
  \end{equation*}
  It follows from \eqref{AInI} that $a\geq 1-\delta b$ and 
  \eqref{FindLambdaBound1} holds.
  
  Finally, we select $b$ to make the right hand side of \eqref{FindLambdaBound1},
  which does not depend on $v$,
  as small as possible.  Then we make choices for $a_{ij}$ according to
  \eqref{ANotInI} and \eqref{AInI}.  We apply Lemma \ref{Calculus} with
  $\alpha=\delta, \beta=2, \gamma=C_F, u=b$ and $v=(1-\delta b)/2$. By the
  lemma, the right hand side of \eqref{FindLambdaBound1} has a unique minimum
  $l$ where $e^{-l}$ is the unique real root larger than $1$ of 
  $x^{-2}+C_Fx^{-\delta} = 1$.  Setting $\rho=e^{-l}$ we
  establish the lemma.
\end{proof}

\noindent {\it Proof of Theorem \ref{main}.}
Suppose $\x\in \mathcal P(\alg)$ and $K$ is a number field containing all coordinates 
of $\x$ and all coefficients of $F$ and has $D_v=2$ for all $v|\infty$.  Assume $a_{ij}, b$ are
the constants from Lemma \ref{FindLambda} and  $\lambda_v$ is defined as in \eqref{lambda}.
Since $x_{ij}$ and $\tilde F$ are multihomogeneous and $n_i+1=\sum_ja_{ij}+b\tilde d_i$, 
Lemma \ref{InitBound} implies that
\begin{equation*}
  \sum_{i=1}^{r}(n_i+1)\log H(\x_i) \geq\sum_{v|\infty}\frac{d_v}{d}\log\lambda_v
\end{equation*}
whenever $x_{ij}\ne 0$ for all $i,j$ and $F(\x^{-1})\ne 0$.
Then by Lemma \ref{FindLambda} we have $\lambda_v\geq\rho$ so that
\begin{equation*}
  \sum_{i=1}^{r}(n_i+1)\log H(\x_i) \geq \sum_{v|\infty}\frac{d_v}{d}\log\rho = \log\rho.
\end{equation*}
\qed

\section{Acknowledgment}
The author wishes to thank Professor J. D. Vaaler for many useful discussions regarding this work.


\begin{thebibliography}{}
\bibitem{BeukersZagier} F. Beukers and D. Zagier, 
  {\it  Lower bounds of heights of points on hypersurfaces},
  Acta Arith.  {\bf 79}  (1997),  no. 2, 103--111.
\bibitem{Dobrowolski} E. Dobrowolski, 
  {\it  On a question of Lehmer and the number of irreducible factors of a polynomial},
  Acta Arith.  {\bf 34} (1979),  no. 4, 391--401.
\bibitem{Lehmer} D.H. Lehmer, {\it Factorization of certain cyclotomic functions}, 
  Ann. of Math. {\bf 34} (1933), 461--479.
\bibitem{Schinzel} A. Schinzel,
  {\it On the product of the conjugates outside the unit circle of an algebraic number},
  Acta Arith. {\bf 24} (1973), 385--399. Addendum, ibid.  {\bf 26} (1975), no. 3, 329--331.
\bibitem{Smyth} C.J. Smyth, 
  {\it  On the product of the conjugates outside the unit circle of an algebraic integer},
  Bull. London Math. Soc.  {\bf 3}  (1971), 169--175.
\bibitem{Zagier}D. Zagier, {\it Algebraic numbers close to both 0 and 1},
  Math. Comp. {\bf 61} (1993), 485--491.
\bibitem{Zhang} S. Zhang, {\it Positive line bundles on arithmetic surfaces},
  Ann. of Math. {\bf 136} (1992), 569--587.
\end{thebibliography}
\end{document}